\documentclass[letterpaper,11pt]{amsart}

\usepackage{amsmath,amssymb,amsthm}

\newcommand{\Z}{\mathbb{Z}}
\newcommand{\R}{\mathbb{R}}
\DeclareMathOperator{\E}{E}
\newcommand{\abs}[1]{\lvert#1\rvert}
\newcommand{\bigabs}[1]{\big\lvert#1\big\rvert}

\newcommand{\biggabs}[1]{\bigg\lvert#1\bigg\rvert}
\newcommand{\Biggabs}[1]{\Bigg\lvert#1\Bigg\rvert}
\newcommand{\e}{\epsilon}

\newtheorem{theorem}{Theorem}[section]
\newtheorem{proposition}[theorem]{Proposition}
\newtheorem{corollary}[theorem]{Corollary}
\newtheorem{lemma}[theorem]{Lemma}
\theoremstyle{definition}
\newtheorem{definition}[theorem]{Definition}

\numberwithin{equation}{section}

\newcommand{\stack}[2]{\genfrac{}{}{0pt}{}{#1}{#2}}

\allowdisplaybreaks[1]

\begin{document}

\title[Random binary sequences]{The peak sidelobe level of random\\binary sequences}

\author{Kai-Uwe Schmidt}

\date{21 February 2011 (revised 21 December 2013)}

\address{Department of Mathematics, 
Simon Fraser University, 8888 University Drive, Burnaby BC V5A 1S6, Canada.}

\curraddr{Faculty of Mathematics, 
Otto-von-Guericke University, Universit\"atsplatz~2, 39106 Magdeburg, Germany.}

\email{kaiuwe.schmidt@ovgu.de}

\thanks{The author was supported by German Research Foundation under Research Fellowship SCHM 2609/1-1.}



\begin{abstract}
Let $A_n=(a_0,a_1,\dots,a_{n-1})$ be drawn uniformly at random from $\{-1,+1\}^n$ and define
\[
M(A_n)=\max_{0<u<n}\,\biggabs{\sum_{j=0}^{n-u-1}a_ja_{j+u}}\quad\mbox{for $n>1$}.
\]
It is proved that $M(A_n)/\sqrt{n\log n}$ converges in probability to $\sqrt{2}$. This settles a problem first studied by Moon and Moser in the 1960s and proves in the affirmative a recent conjecture due to Alon, Litsyn, and Shpunt. It is also shown that the expectation of $M(A_n)/\sqrt{n\log n}$ tends to $\sqrt{2}$.
\end{abstract}

\maketitle


\section{Introduction}

Consider a binary sequence $A=(a_0,a_1,\dots,a_{n-1})$ of length $n$, namely an element of $\{-1,+1\}^n$. Define the \emph{aperiodic autocorrelation} at shift $u$ of $A$ to be
\[
C_u(A)=\sum_{j=0}^{n-u-1}a_ja_{j+u}\quad \mbox{for $u\in\{0,1,\dots,n-1\}$}
\] 
and define the  \emph{peak sidelobe level} of $A$ as
\[
M(A)=\max_{0<u<n}\,\abs{C_u(A)}\quad\mbox{for $n>1$}.
\]
Binary sequences with small autocorrelation at nonzero shifts have a wide range of applications in digital communications, including synchronisation and radar (see~\cite{Golomb2005}, for example).
\par
Let $\mu(n)$ be the minimum of $M(A)$ taken over all $2^n$ binary sequences $A$ of length~$n$. By a parity argument, it is seen that $\mu(n)\ge 1$ and it is known that $\mu(n)=1$ for $n\in\{2,3,4,5,7,11,13\}$ (binary sequences attaining the minimum are often called \emph{Barker} sequences). It is a classical problem to decide whether $\mu(n)>1$ for all $n>13$. Although deep methods have been developed~\cite{Turyn1965},~\cite{Schmidt1999}, this problem is still open; the currently smallest undecided case arises for $n>2\cdot 10^{30}$~\cite{Leung2012}. It is conjectured that $\mu(n)$ grows as $n\to\infty$, perhaps like $\sqrt{n}$. We refer to Turyn~\cite{Turyn1968} and Jedwab~\cite{Jedwab2008} for excellent surveys on this problem.
\par
In this paper, we will be concerned with the asymptotic behaviour, as $n\to\infty$, of $M(A)$ for almost all binary sequences $A$ of length $n$. This problem was first studied by Moon and Moser~\cite{Moon1968}. Let $A_n$ be a random binary sequence of length $n$, by which we mean that $A_n$ is drawn uniformly at random from $\{-1,+1\}^n$. In other words, each of the $n$ sequence elements of $A_n$ takes on each of the values $-1$ and $+1$ independently with probability $1/2$. Until now, the best known bounds are
\begin{equation}
\lim_{n\to\infty}\Pr\bigg[1-\e<\frac{M(A_n)}{\sqrt{n\log n}}<\sqrt{2}+\e\bigg]=1\quad\text{for all $\e>0$}.   \label{eqn:lbup_known}
\end{equation}
The upper bound is due to Mercer~\cite{Mercer2006}. In fact, Mercer proved a weaker result but pointed out in a final remark~\cite[p.~670]{Mercer2006} that his proof establishes the above upper bound. The lower bound was proved by Alon, Litsyn, and Shpunt~\cite{Alon2010}, in response to numerical evidence provided by Dmitriev and Jedwab~\cite{Dmitriev2007}. The authors of~\cite{Alon2010} also conjectured that the lower bound can be improved to $\sqrt{2}-\e$. The aim of this paper is to prove this conjecture and therefore to establish the limit distribution, as $n\to\infty$, of $M(A_n)/\sqrt{n\log n}$. In particular, we prove the following.
\begin{theorem}
\label{thm:main}
Let $A_n$ be a random binary sequence of length $n$. Then, as $n\to\infty$,
\[
\frac{M(A_n)}{\sqrt{n\log n}}              \to\sqrt{2}\quad\mbox{in probability}
\]
and
\[
\frac{\E\big[M(A_n)\big]}{\sqrt{n\log n}}   \to\sqrt{2}.
\]
\end{theorem}
\par
Alon, Litsyn, and Shpunt~\cite{Alon2010} already observed that, as a consequence of McDiarmid's inequality (Lemma~\ref{lem:bounded_differences}), $M(A_n)$ is concentrated around its expected value, but could only show that
\begin{equation}
\liminf_{n\to\infty}\,\frac{\E\big[M(A_n)\big]}{\sqrt{n\log n}}\ge 1.   \label{eqn:lb_E1}
\end{equation}
Their proof considers $C_u(A_n)$ only for $u\ge n/2$ and crucially relies on the fact that $C_u(A_n)$ and $C_v(A_n)$ are independent whenever $n/2\le u<v<n$. Our method considers $C_u(A_n)$ also for $u<n/2$. In particular, by a careful estimation of the moments of $C_u(A_n)C_v(A_n)$ for $0<u<v<n$, we will show that the lower bound~\eqref{eqn:lb_E1} can be improved to $\sqrt{2}$, which together with~\eqref{eqn:lbup_known} establishes the second part of Theorem~\ref{thm:main}. The first part of Theorem~\ref{thm:main} then follows from McDiarmid's inequality.
\par
As pointed out in~\cite{Alon2010}, given a binary sequence $A=(a_0,a_1,\dots,a_{n-1})$ of length $n$, the quantity $M(A)$ is related to the more general $r$th-order correlation measure $S_r(A)$, which was defined by Mauduit and S\'ark\"ozy~\cite{Mauduit1997} to be
\[
S_r(A):=\max_{0\le u_1<u_2<\cdots<u_r<n}\;\max_{0\le k\le n-u_r}\Biggabs{\sum_{j=0}^{k-1}a_{j+u_1}a_{j+u_2}\cdots a_{j+u_r}}\quad\mbox{for $n\ge r$}.
\]
Alon, Kohayakawa, Mauduit, Moreira, and R\"odl~\cite{Alon2007} established that, given a random binary sequence $A_n$ of length $n$, then for all $r\ge 2$,
\[
\lim_{n\to\infty}\Pr\Bigg[\frac{2}{5}<\frac{S_r(A_n)}{\sqrt{n\log {n\choose r}}}<\sqrt{3}+\e\Bigg]=1\quad\text{for all $\e>0$}.
\]
Since, for every binary sequence $A$, we have $M(A)\le S_2(A)$, Theorem~\ref{thm:main} implies that for $r=2$ the lower bound can be improved from $2/5$ to $1-\e$.


\section{Preliminary Results}

The main results of this section are the following. Given a random binary sequence $A_n$ of length $n$, Proposition~\ref{pro:Pr_Cu} gives a lower bound for
\begin{equation}
\Pr\big[\abs{C_u(A_n)}\ge \sqrt{2n\log n}\big]   \label{eqn:lb}
\end{equation}
for small $u$. This result can also be concluded from~\cite{Alon2010}. However, the proof presented here is considerably simpler and more direct. Proposition~\ref{pro:Pr_CuCv} gives an upper bound for
\begin{equation}
\Pr\big[\abs{C_u(A_n)}\ge \sqrt{2n\log n}\,\cap\,\abs{C_v(A_n)}\ge \sqrt{2n\log n}\big]   \label{eqn:ub}
\end{equation}
for $0<u<v<n$. These bounds will be the crucial ingredients to prove the main result of this paper.

\subsection{~}

To bound~\eqref{eqn:lb}, we shall need the following refinement of the central limit theorem.
\begin{lemma}[Cram\'er~{\cite[Thm.~2]{Cramer1938}}]
\label{lem:asymptotic_tail}
Let $X_0,X_1,\dots$ be identically distributed mutually independent random variables satisfying $\E[X_0]=0$ and $\E[X_0^2]=1$ and suppose that there exists $T>0$ such that $\E[e^{tX_0}]<\infty$ for all $\abs{t}<T$. Write $Y_k=X_0+X_1+\cdots+X_{k-1}$ and let $\Phi$ be the distribution function of a normal random variable with zero mean and unit variance. If $\theta_k>1$ and $\theta_k/k^{1/6}\to 0$ as $k\to\infty$, then
\[
\frac{\Pr\big[\abs{Y_k}\ge\theta_k\sqrt{k}\big]}{2\Phi(-\theta_k)}\to 1.
\]
\end{lemma}
\par
\begin{proposition}
\label{pro:Pr_Cu}
Let $A_n$ be a random binary sequence of length $n>2$ and let $u$ be an integer satisfying $1\le u\le\frac{n}{\log n}$. Then
\[
\Pr\big[\abs{C_u(A_n)}\ge \sqrt{2n\log n}\big]\ge\frac{1}{5n\sqrt{\log n}}
\]
for all sufficiently large $n$.
\end{proposition}
\begin{proof}
Write $A_n=(a_0,a_1,\dots,a_{n-1})$. It is well known that the $n-u$ products
\[
a_0a_u,\,a_1a_{1+u},\dots,a_{n-u-1}a_{n-1}
\]
are mutually independent. A proof of this fact was given by Mercer~\cite[Prop.~1.1]{Mercer2006}. Hence $C_u(A_n)$ is a sum of $n-u$ mutually independent random variables, each taking each of the values $-1$ and $+1$ with probability $1/2$. Notice that $\E[e^{ta_0a_u}]=\cosh(t)$ and, setting
\[
\xi_n=\sqrt{\frac{2n\log n}{n-u}},
\]
we find that $\xi_n/(n-u)^{1/6}\to 0$ since $u\le \frac{n}{\log n}$. We can therefore apply Lemma~\ref{lem:asymptotic_tail} to conclude, as $n\to\infty$,
\begin{equation}
\Pr\big[\abs{C_u(A_n)}\ge\sqrt{2n\log n}\big]\sim 2\Phi(-\xi_n),   \label{eqn:Pr_Cu_F}
\end{equation}
where
$\Phi$ is the distribution function of a standard normal random variable. It is well known (see~\cite[Thm.~1.2.3]{Durrett2010}, for example) that
\[
\frac{1}{\sqrt{2\pi}\,z}\Big(1-\frac{1}{z^2}\Big)\,e^{-z^2/2}\le\Phi(-z)\le\frac{1}{\sqrt{2\pi}\,z}\,e^{-z^2/2}\quad\mbox{for $z>0$},
\]
so that, since $\frac{n}{n-u}\sim 1$, as $n\to\infty$,
\[
2\Phi(-\xi_n)\sim \frac{1}{\sqrt{\pi \log n}}\,e^{-\frac{n}{n-u}\log n}.
\]
Using $u\le \frac{n}{\log n}$, we conclude
\[
e^{-\frac{n}{n-u}\log n}\ge e^{-\frac{\log n}{\log n-1}\,\log n}\sim \frac{1}{en}
\]
as $n\to\infty$. It then follows from~\eqref{eqn:Pr_Cu_F} that for all $\alpha>e\sqrt{\pi}$ and all sufficiently large $n$ we have
\[
\Pr\big[\abs{C_u(A_n)}\ge\sqrt{2n\log n}\big]\ge\frac{1}{\alpha n\sqrt{\log n}}.
\]
The lemma follows since $5>e\sqrt{\pi}$.
\end{proof}

\subsection{~}

We now turn to the derivation of an upper bound for~\eqref{eqn:ub}. It will be convenient to define the notion of an even tuple as follows.
\begin{definition}
A tuple $(x_1,x_2,\dots,x_{2m})$ is \emph{even} if there exists a permutation $\sigma$ of $\{1,2,\dots,2m\}$ such that $x_{\sigma(2i-1)}=x_{\sigma(2i)}$ for each $i\in\{1,2,\dots,m\}$.
\end{definition}
\par
For example, $(1,3,1,4,3,4)$ is even, while $(2,1,1,2,1,3)$ is not even. In the next two lemmas we will prove two results about even tuples, which we then use to estimate moments of $C_u(A_n)C_v(A_n)$.
\par
Recall that, for positive integer $k$, the double factorial
\[
(2k-1)!!=\frac{(2k)!}{k!\,2^k}=(2k-1)(2k-3)\cdots 3\cdot 1
\]
is the number of ways to arrange $2k$ objects into $k$ unordered pairs.
\begin{lemma}
\label{lem:even_tuple_single}
Let $m$ and $q$ be positive integers and let $R$ be the set of even tuples in
\[
\big\{(x_1,x_2,\dots,x_{2q})\,:\,x_i\in\Z,\,0\le x_i<m\big\}.
\]
Then 
\[
\abs{R}\le(2q-1)!!\,m^q.
\]
\end{lemma}
\begin{proof}
There are $(2q-1)!!$ ways to arrange $x_1,x_2,\dots,x_{2q}$ into $q$ unordered pairs and to each of these $q$ pairs we assign a value of $\{0,1,\dots,m-1\}$. In this way we construct all elements of $R$ at least once, which proves the lemma.
\end{proof}
\par
\begin{lemma}
\label{lem:even_tuple_double}
Let $u$, $v$, and $n$ be integers satisfying $0<u,v<n$ and $u\ne v$. Write $I=\{1,2,\dots,2q\}$ and let $t$ be an integer satisfying $0\le t<q$. Let $S$ be the subset of 
\[
\big\{(x_i,x_i+u,y_i,y_i+v)_{i\in I}\,:\,x_i,y_i\in\Z,\,0\le x_i<n-u,\,0\le y_i<n-v\big\}
\]
containing all even elements $(x_i,x_i+u,y_i,y_i+v)_{i\in I}$ such that $(x_i)_{i\in J}$ is not even for all $(2q-2t)$-element subsets $J$ of $I$. Then
\[
\abs{S}\le(8q-1)!!\,n^{2q-(t+1)/3}.
\]
\end{lemma}
\begin{proof}
We will construct a set of tuples that contains $S$ as a subset. Arrange the $8q$ variables
\begin{equation}
x_1,x_1+u,\dots,x_{2q},x_{2q}+u,y_1,y_1+v,\dots,y_{2q},y_{2q}+v   \label{eqn:xy}
\end{equation}
into $4q$ unordered pairs $(a_1,b_1),(a_2,b_2),\dots,(a_{4q},b_{4q})$ such that there are at most $q-t-1$ pairs $(x_i,x_j)$. This can be done in at most $(8q-1)!!$ ways. We formally set $a_i=b_i$ for all $i\in\{1,2,\dots,4q\}$. If this assignment does not yield a contradiction, then we call the arrangement of~\eqref{eqn:xy} into $4q$ pairs \emph{consistent}. For example, if there are pairs of the form $(x_i,y_j)$ and $(x_i+u,y_j+v)$, then the arrangement is not consistent since $u\ne v$ by assumption.
\par
Now, for every consistent arrangement, pairs of the form $(x_i,x_j)$ or $(y_i,y_j)$ determine the value of another pair (namely, $(x_i+u,x_j+u)$ or $(y_i+v,y_j+v)$, respectively). On the other hand, for every consistent arrangement, pairs not of the form
\[
(x_i,x_j),\;(y_i,y_j),\;(x_i+u,x_j+u),\;\mbox{or}\; (y_i+v,y_j+v)
\]
determine the value of at least two other pairs. For example, if there exists the pair $(x_i,y_j)$, then $x_i+u$ and $y_j+v$ must lie in different pairs. Therefore, since there are at most $q-t-1$ pairs of the form $(x_i,x_j)$ and at most $q$ pairs of the form $(y_i,y_j)$, for each consistent arrangement, at most 
\[
\tfrac{1}{2}(4q-2t-2)+\tfrac{1}{3}(2t+2)=2q-\tfrac{1}{3}(t+1)
\]
of the variables $x_1,\dots,x_{2q},y_1,\dots,y_{2q}$ can be chosen independently. We assign to each of these a value of $\{0,1,\dots,n-1\}$. In this way, we construct a set of at most $(8q-1)!!\,n^{2q-(t+1)/3}$ tuples that contains $S$ as a subset, as required.
\end{proof}
\par
We now use Lemmas~\ref{lem:even_tuple_single} and~\ref{lem:even_tuple_double} to bound moments of $C_u(A_n)C_v(A_n)$.
\begin{lemma}
\label{lem:moments}
Let $p$ and $h$ be integers satisfying $0\le h<p$ and let $A_n$ be a random binary sequence of length $n$. Then, for $0<u<v<n$,
\[
\E\Big[\big(C_u(A_n)C_v(A_n)\big)^{2p}\Big]\le n^{2p}\big[(2p-1)!!\big]^2\bigg(1+\frac{(8p)^{8h}}{n^{1/3}}+\frac{(8p)^{4p}}{n^{(h+1)/3}}\bigg).
\]
\end{lemma}
\begin{proof}
Write $I=\{1,2,\dots,2p\}$ and let $T$ be the set containing all even tuples of
\[
\big\{(x_i,x_i+u,y_i,y_i+v)_{i\in I}\,:\,x_i,y_i\in\Z,\,0\le x_i<n-u,\,0\le y_i<n-v\big\}.
\]
Writing $A_n=(a_0,a_1,\dots,a_{n-1})$, we have
\begin{align}
\lefteqn{\E\Big[\big(C_u(A_n)C_v(A_n)\big)^{2p}\Big]}   \nonumber\\
&=\E\Bigg[\bigg(\sum_{i=0}^{n-u-1}a_ia_{i+u}\bigg)^{2p}\bigg(\sum_{j=0}^{n-v-1}a_ja_{j+v}\bigg)^{2p}\Bigg]   \nonumber\\
&=\sum_{i_1,\dots,i_{2p}=0}^{n-u-1}\;\sum_{j_1,\dots,j_{2p}=0}^{n-v-1}\E\big[a_{i_1}a_{i_1+u}\cdots a_{i_{2p}}a_{i_{2p}+u}a_{j_1}a_{j_1+v}\cdots a_{j_{2p}}a_{j_{2p}+v}\big]   \nonumber\\[1ex]
&=\abs{T}   \label{eqn:moment_T}
\end{align}
since $a_0,a_1,\dots,a_{n-1}$ are mutually independent, $\E[a_j]=0$, and $a_j^2=1$ for all $j\in\{0,1,\dots,n-1\}$. %
We define the following subsets of $T$.
\begin{enumerate}
\item $T_1$ contains all elements $(x_i,x_i+u,y_i,y_i+v)_{i\in I}$ of $T$ such that $(x_i)_{i\in I}$ and $(y_i)_{i\in I}$ are even.
\item $T_2$ contains all elements $(x_i,x_i+u,y_i,y_i+v)_{i\in I}$ of $T$ such that $(x_i)_{i\in I}$ or $(y_i)_{i\in I}$ is not even and $(x_i)_{i\in J}$ and $(y_i)_{i\in K}$ are even for some $(2p-2h)$-element subsets $J$ and $K$ of $I$.
\item $T_3$ contains all elements $(x_i,x_i+u,y_i,y_i+v)_{i\in I}$ of $T$ such that either $(x_i)_{i\in J}$ is not even for all $(2p-2h)$-element subsets $J$ of $I$ or $(y_i)_{i\in K}$ is not even for all $(2p-2h)$-element subsets $K$ of $I$.
\end{enumerate}
It is immediate that $T_1$, $T_2$, and $T_3$ partition $T$, so that
\begin{equation}
\abs{T}=\abs{T_1}+\abs{T_2}+\abs{T_3}.   \label{eqn:T1T2T3}
\end{equation}
We now bound the cardinalities of $T_1$, $T_2$, and $T_3$.
\par
{\itshape The set $T_1$.} Using Lemma~\ref{lem:even_tuple_single}, we have the crude estimate
\begin{equation}
\abs{T_1}\le \big[(2p-1)!!\big]^2\,n^{2p}.   \label{eqn:T1}
\end{equation}
\par
{\itshape The set $T_2$.} Let $(x_i,x_i+u,y_i,y_i+v)_{i\in I}$ be an element of $T_2$. Then there exist $(2p-2h)$-element subsets $J$ and $K$ of $I$ such that $(x_i)_{i\in J}$ and $(y_i)_{i\in K}$ are even and
\begin{equation}
(x_i)_{i\in I\setminus J}\quad\mbox{or}\quad(y_i)_{i\in I\setminus K}   \label{eqn:tuple_2h_lr}
\end{equation}
is not even. Since $(x_i)_{i\in J}$ and $(y_i)_{i\in K}$ are even, $(x_i,x_i+u,y_j,y_j+v)_{i\in J,\,j\in K}$ is even. Since $(x_i,x_i+u,y_i,y_i+v)_{i\in I}$ is also even, it follows that
\begin{equation}
(x_i,x_i+u,y_j,y_j+v)_{i\in I\setminus J,\,j\in I\setminus K}   \label{eqn:tuple_2h}
\end{equation}
is even as well. There are ${2p\choose 2h}$ subsets $J$ and ${2p\choose 2h}$ subsets $K$. By Lemma~\ref{lem:even_tuple_single}, for each such $J$ and $K$, there are at most $(2p-2h-1)!!\,n^{p-h}$ even tuples $(x_i)_{i\in J}$ satisfying $0\le x_i<n$ for each $i\in J$ and at most $(2p-2h-1)!!\,n^{p-h}$ even tuples $(y_i)_{i\in K}$ satisfying $0\le y_i<n$ for each $i\in K$. By Lemma~\ref{lem:even_tuple_double} applied with $t=0$ and by interchanging $u$ and $v$ and $(x_i)_{i\in I\setminus J}$ and $(y_i)_{i\in I\setminus K}$ if necessary, the number of even tuples in $\{0,1,\dots,n-1\}^{8h}$ of the form~\eqref{eqn:tuple_2h} such that one of the tuples in~\eqref{eqn:tuple_2h_lr} is not even is at most $(8h-1)!!\,n^{2h-1/3}$. Therefore,
\begin{align}
\abs{T_2}&\le 2n^{2h-1/3}\,(8h-1)!!\bigg[{2p\choose 2h}(2p-2h-1)!!\,n^{p-h}\bigg]^2   \nonumber\\[1ex]
&\le n^{2p-1/3}\big[(2p-1)!!\big]^2\,(8p)^{8h},   \label{eqn:T2}
\end{align}
using very crude bounds.
\par
{\itshape The set $T_3$.} By Lemma~\ref{lem:even_tuple_double} applied with $t=h$ and by interchanging $u$ and $v$ and $(x_i)_{i\in I}$ and $(y_i)_{i\in I}$ if necessary,
\begin{align}
\abs{T_3}&\le2 n^{2p-(h+1)/3}\,(8p-1)!!   \nonumber\\
&\le n^{2p-(h+1)/3}\,(8p)^{4p}.   \label{eqn:T3}
\end{align}
\par
Now the lemma follows by combining~\eqref{eqn:moment_T},~\eqref{eqn:T1T2T3},~\eqref{eqn:T1},~\eqref{eqn:T2}, and~\eqref{eqn:T3}.
\end{proof}
\par
Lemma~\ref{lem:moments} is now used to prove the desired upper bound for~\eqref{eqn:ub}.
\begin{proposition}
\label{pro:Pr_CuCv}
Let $A_n$ be a random binary sequence of length $n$ and write $\lambda_n=\sqrt{2n\log n}$. Then, for $0<u<v<n$ and all sufficiently large $n$,
\[
\Pr\big[\abs{C_u(A_n)}\ge \lambda_n\,\cap\,\abs{C_v(A_n)}\ge \lambda_n\big]\le \frac{23}{n^2}.
\]
\end{proposition}
\begin{proof}
Let $(X_1,X_2)$ be a random vector taking values in $\R\times\R$ and let $p$ be a positive integer. Then by Markov's inequality, for $\theta_1,\theta_2>0$,
\[
\Pr\big[\abs{X_1}\ge\theta_1\,\cap\,\abs{X_2}\ge\theta_2\big]\le\frac{\E\big[(X_1X_2)^{2p}\big]}{(\theta_1\theta_2)^{2p}}.
\]
Let $h$ be an arbitrary integer satisfying $0\le h<p$. Application of Lemma~\ref{lem:moments} gives
\begin{multline}
\Pr\big[\abs{C_u(A_n)}\ge\lambda_n\,\cap\,\abs{C_v(A_n)}\ge\lambda_n\big)\\
\le\frac{[(2p-1)!!\big]^2}{(2\log n)^{2p}}\big[1+K_1(n,p,h)+K_2(n,p,h)\big],   \label{Pr_moment}
\end{multline}
where
\begin{align*}
K_1(n,p,h)=n^{-1/3}\,(8p)^{8h}\quad\mbox{and}\quad
K_2(n,p,h)=n^{-(h+1)/3}\,(8p)^{4p}.
\end{align*}
We apply~\eqref{Pr_moment} with $p=\lfloor\log n\rfloor$ and $h=\lfloor 17\log\log n\rfloor$, so that for all sufficiently large $n$ we have $h<p$, as assumed. By Stirling's approximation
\[
\sqrt{2\pi k}\;k^ke^{-k}\le k!\le\sqrt{3\pi k}\;k^ke^{-k},
\]
we have
\[
\frac{[(2p-1)!!\big]^2}{(2\log n)^{2p}}\le \frac{3p^{2p}e^{-2p}}{(\log n)^{2p}}\le\frac{3e^2}{n^2}.
\]
We also have
\begin{align*}
K_1(n,p,h)&\le K_1(n,\log n,17\log\log n)\\
&=n^{-\frac{1}{3}}n^{\frac{136(\log\log n)(\log 8+\log\log n)}{\log n}}&\\
&=O(n^{-1/4})\quad\mbox{as $n\to\infty$}
\intertext{and}
K_2(n,p,h)&\le K_2(n,\log n,16\log\log n)\\
&=n^{-\frac{1}{3}+4\log 8-\frac{4}{3}\log\log n}\\
&=O(n^{-\log\log n})\quad\mbox{as $n\to\infty$}.
\end{align*}
Substitute into~\eqref{Pr_moment} to obtain the claimed result, using $3e^2<23$.
\end{proof}

\section{Proof of Main Theorem}

We require the following result, which is a consequence of Azuma's inequality for martingales.
\begin{lemma}[{McDiarmid~\cite{McDiarmid1989}}] \label{lem:bounded_differences} 
Let $X_0,X_1,\dots,X_{n-1}$ be mutually independent random variables taking values in a set $S$. Let $f:S^n\to\R$ be a measurable function and suppose that $f$ satisfies
\[
\bigabs{f(x)-f(y)}\le c
\]
whenever $x$ and $y$ differ only in one coordinate. Define the random variable $Y=f(X_0,X_1,\dots,X_{n-1})$. Then, for $\theta\ge 0$,
\[
\Pr\big[\bigabs{Y-\E[Y]}\ge \theta\big]\le 2e^{-\frac{2\theta^2}{c^2n}}.
\]
\end{lemma}
\par
Given a random binary sequence $A_n=(a_0,a_1,\dots,a_{n-1})$ of length~$n$, we will apply Lemma~\ref{lem:bounded_differences} with $X_j=a_j$ for $j\in\{0,1,\dots,n-1\}$ and
\[
f(x_0,x_1,\dots,x_{n-1})=\max_{0<u<n}\Biggabs{\sum_{j=0}^{n-u-1}x_jx_{j+u}},
\]
so that $M(A_n)=f(a_0,a_1,\dots,a_{n-1})$. We can take $c=4$ in Lemma~\ref{lem:bounded_differences} and obtain the following corollary.
\begin{corollary}
\label{cor:Inequality_MA}
Let $A_n$ be a random binary sequence of length $n$. Then, for $\theta\ge 0$,
\[
\Pr\big[\bigabs{M(A_n)-\E[M(A_n)]}\ge \theta\big]\le 2e^{-\frac{\theta^2}{8n}}.
\]
\end{corollary}
\par
We now prove the second part of Theorem~\ref{thm:main}.
\begin{theorem}
\label{thm:M_mean}
Let $A_n$ be a random binary sequence of length $n$. Then, as $n\to\infty$,
\[
\frac{\E\big[M(A_n)\big]}{\sqrt{n\log n}}\to \sqrt{2}.
\]
\end{theorem}
\begin{proof}
By the triangle inequality and the union bound we have, for all $\e>0$,
\begin{multline*}
\Pr\bigg[\frac{\E\big[M(A_n)\big]}{\sqrt{n\log n}}-\sqrt{2}>\e\bigg]\\
\le \Pr\bigg[\frac{\E\big[M(A_n)\big]}{\sqrt{n\log n}}-\frac{M(A_n)}{\sqrt{n\log n}}>\tfrac{1}{2}\e\bigg]+\Pr\bigg[\frac{M(A_n)}{\sqrt{n\log n}}-\sqrt{2}>\tfrac{1}{2}\e\bigg]. 
\end{multline*}
By Corollary~\ref{cor:Inequality_MA} and the upper bound of~\eqref{eqn:lbup_known}, the two terms on the right-hand side tend to zero as $n\to\infty$, hence
\begin{equation}
\limsup_{n\to\infty}\;\frac{\E\big[M(A_n)\big]}{\sqrt{n\log n}}\le\sqrt{2}.   \label{eqn:limsup_E}
\end{equation}
Let $\delta>0$ and define the set
\begin{equation}
N(\delta)=\bigg\{n>1:\frac{\E\big[M(A_n)\big]}{\sqrt{n\log n}}<\sqrt{2}-\delta\bigg\}.   \label{eqn:def_N_delta}
\end{equation}
We claim that the size of $N(\delta)$ is finite for all choices of $\delta$, which together with~\eqref{eqn:limsup_E} will prove the theorem. The proof of the claim is based on an idea developed in~\cite{Alon2010}. Let $n>2$ and write
\[
W=\Big\{u\in\Z:1\le u\le\frac{n}{\log n}\Big\}
\]
and $\lambda_n=\sqrt{2n\log n}$. Then
\begin{multline*}
\Pr\big[M(A_n)\ge\lambda_n\big]\ge\Pr\big[\max_{u\in W}\;\abs{C_u(A_n)}\ge\lambda_n\big]\\
\ge \sum_{u\in W}\Pr\big[\abs{C_u(A_n)}\ge\lambda_n\big]-\sum_{\stack{u,v\in W}{u<v}}\Pr\big[\abs{C_u(A_n)}\ge\lambda_n\,\cap\,\abs{C_v(A_n)}\ge\lambda_n\big]
\end{multline*}
by the Bonferroni inequality. By Propositions~\ref{pro:Pr_Cu} and~\ref{pro:Pr_CuCv},
\begin{align}
\Pr\big[M(A_n)\ge\lambda_n\big]&\ge\abs{W}\cdot\frac{1}{5n(\log n)^{\frac{1}{2}}}-\frac{\abs{W}^2}{2}\cdot\frac{23}{n^2}   \nonumber\\
&\ge\frac{1}{8(\log n)^{\frac{3}{2}}}-\frac{12}{(\log n)^2}   \nonumber\\
&\ge \frac{1}{10(\log n)^{\frac{3}{2}}}   \label{eqn:PrMA_lb}
\end{align}
for all sufficiently large $n$, using $\frac{2}{3}\frac{n}{\log n}\le\abs{W}\le\frac{n}{\log n}$ for $n>2$. Now, by the definition~\eqref{eqn:def_N_delta} of $N(\delta)$, for all $n\in N(\delta)$ we have $\lambda_n>\E[M(A_n)]$, so that we can apply Corollary~\ref{cor:Inequality_MA} with $\theta=\lambda_n-\E[M(A_n)]$ to give, for all $n\in N(\delta)$,
\[
\Pr\big[M(A_n)\ge\lambda_n\big]\le 2e^{-\frac{1}{8n}(\lambda_n-\E[M(A_n)])^2}.
\]
Comparison with~\eqref{eqn:PrMA_lb} yields, for all sufficiently large $n\in N(\delta)$,
\[
\frac{1}{10(\log n)^{\frac{3}{2}}}\le 2e^{-\frac{1}{8n}(\lambda_n-\E[M(A_n)])^2},
\]
which implies
\[
\frac{\E\big[M(A_n)\big]}{\sqrt{n\log n}}\ge \sqrt{2}-\sqrt{\frac{12\log\log n+8\log 20}{\log n}}.
\]
From the definition~\eqref{eqn:def_N_delta} of $N(\delta)$ it then follows that $N(\delta)$ has finite size for all $\delta>0$, as required.
\end{proof}
\par
Using Corollary~\ref{cor:Inequality_MA}, it is now straightforward to prove the first part of Theorem~\ref{thm:main}.
\begin{corollary}
Let $A_n$ be a random binary sequence of length $n$. Then, as $n\to\infty$,
\[
\frac{M(A_n)}{\sqrt{n\log n}}\to\sqrt{2}\quad\mbox{in probability}.
\]
\end{corollary}
\begin{proof}
By the triangle inequality and the union bound we have, for all $\e>0$,
\begin{multline*}
\Pr\Bigg[\biggabs{\frac{M(A_n)}{\sqrt{n\log n}}-\sqrt{2}}>\e\Bigg]\\
\le \Pr\Bigg[\biggabs{\frac{M(A_n)}{\sqrt{n\log n}}-\frac{\E\big[M(A_n)\big]}{\sqrt{n\log n}}}>\tfrac{1}{2}\e\Bigg]+\Pr\Bigg[\biggabs{\frac{\E\big[M(A_n)\big]}{\sqrt{n\log n}}-\sqrt{2}}>\tfrac{1}{2}\e\Bigg].
\end{multline*}
By Corollary~\ref{cor:Inequality_MA} and Theorem~\ref{thm:M_mean}, the two terms on the right-hand side tend to zero as $n\to\infty$, which proves the corollary.
\end{proof}


\section*{Acknowledgement}
I would like to thank Jonathan Jedwab for many valuable discussions and Jonathan Jedwab and Daniel J.\ Katz for their careful comments on this paper.


\providecommand{\bysame}{\leavevmode\hbox to3em{\hrulefill}\thinspace}
\providecommand{\MR}{\relax\ifhmode\unskip\space\fi MR }
\providecommand{\MRhref}[2]{%
  \href{http://www.ams.org/mathscinet-getitem?mr=#1}{#2}
}
\providecommand{\href}[2]{#2}


\end{document}